\documentclass[12pt]{amsart}

  \usepackage{graphicx}
  \usepackage{color}
  \usepackage{boxedminipage}
  \usepackage{pictexwd,dcpic}
  \usepackage{latexsym,amscd,amssymb,epsfig,a4wide}

  \DeclareMathOperator{\Hilb}{{\rm Hilb}}

  \newcommand{\al }{\alpha }

  \newcommand{\cI }{\mathcal{I}}

  \newcommand{\dlo }{<_{\mathrm{dlex}}}
  \newcommand{\dloeq}{\leq_{\mathrm{dlex}}}
  \newcommand{\dos }{\mathrm{OS}}

  \newcommand{\GB }[1]{\mathcal{I}_q(#1)}

  \newcommand{\ndN }{\mathbb{N}}

  \newcommand{\ndZ }{\mathbb{Z}}

  \newcommand{\rk }{\mathrm{rk}}
  
  \newcommand{\redO }[1][\emptyset]{\reduc{\GB{#1}}}
  \newcommand{\reduc }[1]{\searrow \!\!\!^{#1}}

  \newcommand{\sso }{\sqsubseteq }
  \newcommand{\tm }{t^-}
  \newcommand{\tp }{t^+}

  \def\CC{{\mathbb C}}

  \def\rk{{\mathrm{rk}}}
  \def\Mat{{\mathfrak{M}}}
  \def\gMat{{\mathfrak{N}}}
  \def\OMat{\overline{\Mat }}
  \def\Sord{{<}}
  \def\Sordeq{{\leq}}
  \def\Sordeq{{\le }}
  \def\lm{{\mathrm{lm}}}
  \def\lc{{\mathrm{lc}}}
  \def\lt{{\mathrm{lt}}}

  \theoremstyle{plain}
    \newtheorem{theorem}{Theorem}[section]
    \newtheorem{proposition}[theorem]{Proposition}
    \newtheorem{lemma}[theorem]{Lemma}
    \newtheorem{corollary}[theorem]{Corollary}
    
  \theoremstyle{definition}
    
    \newtheorem{example}[theorem]{Example}

    \newtheorem{remark}[theorem]{Remark}

  \numberwithin{equation}{section}

  \newenvironment{note}[1][Note]
   {\bigskip\begin{center}\begin{boxedminipage}{4.5in}\setlength{\parindent}{1em}\noindent\textbf{#1.
}}
   {\end{boxedminipage}\end{center}\bigskip}

\begin{document}

\title{A deformation of the Orlik-Solomon algebra}

  \author{Istv\'an Heckenberger}
  \address{Fachbereich Mathematik und Informatik\\
           Philipps-Universit\"at Marburg\\
           35032 Marburg, Germany}
  \email{heckenberger@mathematik.uni-marburg.de}

  \author{Volkmar~Welker}
  \email{welker@mathematik.uni-marburg.de}

  \begin{abstract}
    A deformation of the Orlik-Solomon algebra of a matroid $\Mat$ is
    defined as a quotient of the free associative algebra over a 
    commutative ring $R$ with $1$. It is shown that the given generators 
    form a Gr\"obner basis and that after suitable homogenization 
    the deformation and the Orlik-Solomon have the same Hilbert series as $R$-algebras.
    For supersolvable matroids, equivalently fiber type arrangements, there is a 
    quadratic Gr\"obner basis and hence the algebra
    is Koszul.
  \end{abstract}

  \subjclass[2000]{Primary: 05B35 Secondary: 16S37, 16S80}
 
  \maketitle

  \section{Introduction and Statement of Results}

  In this paper we introduce and study a deformation of the Orlik-Solomon algebra 
  of a matroid $\Mat$. We refer the reader to \cite{inp-Orlik07} and \cite{a-Yuzvinsky01} for general facts about
  the classical Orlik-Solomon algebra. Our deformation, which is different from the one in \cite{a-sheltonyuzvinsky97},
  is presented as a quotient of the free
  associative algebra over some commutative ring $R$ with $1$ by an ideal 
  $I_q(\Mat)$ whose generators are 
  deformations of the classical generators of the defining ideal of the
  Orlik-Solomon algebra by a parameter $q \in R$. Choosing $q = 0$ yields the
  Orlik-Solomon algebra over $R$. Our main result, Theorem \ref{thm:main}, states that the
  given generators of $I_q(\Mat)$ are a Gr\"obner basis of the ideal. As a consequence
  it is shown in Corollary \ref{cor:hilbert} that the deformation with $q$ regarded 
  as a degree $2$ element is a standard graded $R$-algebra which has the same 
  Hilbert series as the Orlik-Solomon algebra. For supersolvable matroids, equivalently fiber
  type arrangements, the 
  existence of a quadratic Gr\"obner basis is shown which implies that the algebra is Koszul.
  As further consequences we 
  obtain in Corollary \ref{cor:os} a known Gr\"obner basis for the Orlik-Solomon algebra as
  a quotient of the free and the exterior algebra.  
  The remaining part of the introduction is devoted to the basic definitions and statement of
  results. In Section \ref{sec:groebner} basic facts about non-commutative 
  Gr\"obner basis theory are given. Section \ref{sec:technical} provides technical lemmas
  needed for the proof of the main result. Finally in Section \ref{sec:proofs} the
  missing proofs are given and an independence statement is presented.

  Let $S$ be a finite set and fix a total order $\Sord$ on $S$. 
  Let $R$ be a commutative ring with unit $1$ and let $q \in R$.
  For an arbitrary set system $\Mat \subseteq 2 ^S$ we define a two-sided ideal $I_q(\Mat)$ 
  in the ring $$A := R\langle t_s~|~s \in S\rangle$$
  of non-commutative polynomials  
  in the variables $t_s$, $s \in S$, with coefficients in $R$.

  Let $\GB{\Mat }$ denote the subset of $A$ consisting of the elements
  \begin{gather}
     t_s^2-q,\quad s\in S, \label{eq:ts2}\\
     t_rt_s+t_st_r-2q,\quad r,s\in S, s\Sord r \label{eq:trts},\\
     \tm_J :=\sum _{I\subseteq J,2\nmid \#I} (-1)^{\ell _J(I)}
       (-q)^{(\#I-1)/2}t_{J\setminus I}, \quad \mbox{~for all~} J\in \Mat. \label{eq:tJ}
  \end{gather}

  Here, for a subset $I = \{ j_{\alpha_1} \Sord \cdots \Sord j_{\alpha_{\# I}} \}
  \subseteq J=\{j_1\Sord \cdots \Sord j_{\#J}\}$ 
  we set
  $$\ell_J(I)=\sum _{\nu=1}^{\#I}(\al _{\nu}-\nu ).$$
  The two-sided ideal of $A$ generated by $\GB{\Mat}$ will be denoted by
  $I_q(\Mat )$.
  We write $\dos_q(\Mat)$ for the quotient $A / I_q(\Mat)$. 
  Our main motivation comes from the situation when $q = 0$ and $\Mat$ is indeed the set of 
  circuits of a loopless matroid without parallel elements. 
  We refer the reader to the books \cite{b-Welsh} and
  \cite{b-Oxley} as a general reference for matroid theory and recall that for loopless matroids without parallel element set of circuits $\Mat$ 
  is characterized by the 
  following three axioms:
  \begin{itemize}
    \item[(C1)] $J\in \Mat$ $\Rightarrow$ $\# J>2$.
    \item[(C2)] $J,K\in \Mat$, $J\subseteq K$ $\Rightarrow $ $J=K$.
    \item[(C3)] For any $J,K\in \Mat$ such that $J\not=K$
      and for any $x\in J\cap K$ there exists $L\in \Mat$ such that $L\subseteq
    (J\cup K)\setminus \{x\}$.
  \end{itemize}
  We will refer to (C3) also by the name \textit{circuit axiom}. 

  Now if $\Mat$ is the set of circuits of a loopless matroid without parallel elements then for $q = 0$ the algebra
  $\dos_q(\Mat )$ is the \textit{Orlik-Solomon algebra} of $\Mat$.
  In case $\Mat$ is realizable as the set of circuits of a finite set of hyperplanes in 
  $\CC^d$ then the Orlik-Solomon algebra of $\Mat$ is known to be the cohomology 
  algebra of the set-theoretic complement of the union of the hyperplanes \cite[(5.2)]{a-OrlikSolomon80}.
  In general, for $q\not=0$ the algebra $\dos _q(\Mat )$ is not isomorphic to the
  Orlik-Solomon algebra of $\Mat $, take for example $S$ with $\#S=1$.
  On the other hand, if $q$ is a square in $R$, then $\dos _q(\Mat )$ is easily seen to be
  isomorphic to $\dos _1(\Mat )$.

  Before we can proceed to the statement of our main results we need some more definitions.  
  Throughout this paper we will use the \textit{degree lexicographic} order $\dlo$ (see Section \ref{sec:basics}) on the monomials
  in $A$ induced by the total order $<$ on $S$ as a term order for the monomials in $A$. 
  We enumerate the elements of any subset $J\subseteq S$ by $j_1,\dots ,j_{\#J}$ such that
  $j_1\Sord \cdots \Sord j_{\#J}$. For any $J\subseteq S$ we write $t_J$ for 
  the monomial $t_{j_1} \cdots t_{j_{\#J}}\in A$. 
  The degree lexicographic order on the monomials $t_J$ induces the degree lexicographic order on subsets of $S$;
  that is for subsets $J,K \subseteq S$ we set $J \dlo K$ if and only if $\#J<\#K$ or $\#J=\#K$ and the minimum 
  of the symmetric difference of $J$ and $K$ is contained in $J$.
  In addition, for two subsets $K,J \subseteq S$ we say that $K$ is a \textit{convex subset of} $J$ if $K\subseteq J$
  and if $j\in K$ for all $j\in J$ with $k\Sord j\Sord k'$ for some $k,k'\in K$.
  We write $K\sso J$ in this situation.
  Recall that a subset $J \subseteq S$ is called \textit{dependent} in $\Mat$
  if it contains a circuit from $\Mat$. 
  Let $\OMat $ be the set of dependent subsets of $S$ defined recursively as follows:
  \begin{itemize}
    \item[(GC)] A dependent set $J\subseteq S$ belongs to $\OMat $ if and only 
     if $K\in \OMat $ with $K\dlo J$ implies that $K\setminus \{k_1\}\not \sqsubseteq J\setminus \{j_1\}$.
  \end{itemize}
  Then we call $\OMat $ the set of \textit{Gr\"obner circuits}
  of $\Mat $. 

  \begin{theorem}
    \label{thm:main}
    Let $\OMat $ be a set of Gr\"obner circuits of a loopless matroid $\Mat $ without parallel elements.
    Then $I_q(\Mat) = I_q(\OMat)$ and the set $\GB{\OMat}$ is a 
    Gr\"obner basis of the ideal $I_q(\Mat)$ with respect to the degree lexicographic order for all choices of $q$. 
  \end{theorem} 

  The Gr\"obner basis is easily seen to depend on the total order chosen on $S$ but in Proposition 
  \ref{pr:indep} we show that $I_q(\Mat)$ and hence $\dos_q(\Mat)$ is independent of the order on $S$.  
  Simple inspection shows that the leading monomial of \eqref{eq:ts2} is $t_s^2$, of \eqref{eq:trts} is $t_rt_s$ for
  $s \Sord r$ and of \eqref{eq:tJ} is $t_{J \setminus \{ j_1 \}}$.
  Thus all leading coefficients of $\GB{\OMat}$ are $1$. 
  The standard monomials with respect to the Gr\"obner basis $\GB{\OMat}$ are the monomials $m = t_J$
  for $J \subseteq S$ for which there is no factorization $m = m_1 t_{K'} m_2$ for monomials $m_1, m_2$ where $K'$ is a \textit{broken Gr\"obner circuit}; that
  is there is a Gr\"obner circuit $K$ for which $K' = K \setminus \{k_1\}$.
  By the definition of a Gr\"obner circuit it then follows that the standard monomials with respect to the Gr\"obner basis $\GB{\OMat}$ are the
  monomials $m = t_J$ such that $J$ does not contain a \textit{broken circuit} of $\Mat$; that is a circuit with its least element removed.
  These facts immediately imply:

  \begin{corollary} 
    \label{cor:freemodule}
    Let $\Mat$ be a loopless matroid without parallel elements. 
    Then the algebra $\dos_q(\Mat)$ is a free 
    $R$-module whose rank is independent of $q$.

    The standard monomials with respect to the Gr\"obner basis $\GB{\OMat}$ are the monomials $m = t_J$ 
    for $J \subseteq S$ for which $J$ does not contain a broken circuit.
  \end{corollary}
  
  The algebra $\dos _q(\Mat)$ is $\ndZ /2\ndZ$-graded for the grading induced by $\deg t_s=1$
  for all $s\in S$.
  If $R$ is $\ndZ$-graded and $t_0\in R$ is homogeneous of degree one 
  then the $\ndZ$-grading
  of $R$ extends to a $\ndZ$-grading of $\dos _q(\Mat)$ for $q = t_0^2$.
  We consider the case $R = Q[t_0]$ for some commutative ring $Q$ with $1$ and extend the total order on the variables by setting 
  $t_0$ to be the least variable. Then we consider $I_q(\Mat)$ as an ideal in $Q \langle t_s | s\in S \rangle [t_0]$. 
  We deduce from  Theorem \ref{thm:main}, standard facts about homogenizing Gr\"obner bases  (see \cite[Thm. 3.7]{b-Li02})  
  and Corollary \ref{cor:freemodule} the following corollary.

 \begin{corollary} 
    \label{cor:hilbert}
    Let $\Mat$ be a loopless matroid without parallel elements,
    $R = Q[t_0]$ for a commutative ring $Q$ with $1$ and $t_0$ a degree one variable.
    If we set $q = t_0^2$ then 
    \begin{enumerate}
      \item 
        the set $\GB{\OMat}$ is a Gr\"obner basis of $I_q(\Mat)$ for the degree lexicographic order with $t_0$ being the
        least variable. The algebra $\dos_q(\Mat)$ is a free $Q$-module and a standard graded $Q$-algebra. 
      \item The Hilbert series of $\dos_q(\Mat)$ as a $Q$-algebra is
        $$\frac{1+c_1 z + \dots + c_{\rk(\Mat)} z^{\rk(\Mat)}}{1-z} = \Hilb_\Mat(z) \cdot \frac{1}{1-z} ,$$
        where $\rk(\Mat)$ is the rank of $\Mat$, $c_i$ is the number of subsets of $S$ of cardinality $i$ not
        containing a broken circuit and $\Hilb_\Mat(z)$ the Hilbert series of the Orlik-Solomon algebra of $\Mat$
        as an algebra over a field $k$.
    \end{enumerate}
  \end{corollary}

  Note that the first part together with the second part of Corollary \ref{cor:freemodule} implies that the standard monomials of the Gr\"obner basis are
  $q^i t_J$ for $i \geq 0$ and $J \subset S$ such that $J$ does not contain a broken circuit. Part (2) of Corollary \ref{cor:hilbert} 
  hence follows by a simple counting argument and standard facts about Orlik-Solomon algebras.
  
  We note that experiments suggest that the generators \eqref{eq:ts2}, \eqref{eq:trts} and
  \eqref{eq:tJ} are the unique deformations of the corresponding polynomials for $q = 0$ by variables 
  of degree $\geq 1$ satisfying Corollary \ref{cor:hilbert}(2).

  Using results from matroid theory \cite{a-bjoernerziegler91} we obtain the following results 
  extending results from \cite{a-sheltonyuzvinsky97} (Koszul property) and \cite{a-peeva03} (quadratic 
  Gr\"obner basis and Koszul property) for Orlik-Solomon algebras to our deformation.
  We refer the reader to \cite{inp-froeberg99} for basic facts about Koszul algebras.
  
  \begin{corollary}
    \label{cor:koszul}
    Let $\Mat$ be a supersolvable loopless matroid without parallel elements,
    $R = Q[t_0]$ for a commutative ring $Q$ with $1$ and $t_0$ a degree one variable.
    If we set $q = t_0^2$ then $\GB{\OMat}$ is a quadratic Gr\"obner basis and 
    in particular $\dos_q(\Mat)$ is a standard graded Koszul algebra.
  \end{corollary}

  We postpone the derivation of this corollary till Section \ref{sec:proofs}.
  For $q = 0$, Theorem \ref{thm:main} states

  \begin{corollary}
    \label{cor:os}
    Let $\OMat $ be the set
    of Gr\"obner circuits of a loopless matroid $\Mat$ without parallel elements.
    Then the polynomials $t_s^2$ with $s \in S$, $t_rt_s+t_st_r$ with $r,s \in S$, $s\Sord r$ and
    $\sum_{\nu = 1}^{\#J} (-1)^{\nu -1} t_{J \setminus \{j_{\nu }\}}$, where
    $J \in \OMat $, form
    a Gr\"obner basis of the defining ideal $I_0(\Mat)$ of the Orlik-Solomon
    algebra of $\Mat$ with respect to the degree lexicographic order.
  \end{corollary}

  Since for $q = 0$ the quotient of $A$ by $t_s^2$ with $s \in S$, $t_rt_s+t_st_r$ with $r,s \in S$, $s\Sord r$
  is the exterior algebra $E$ we also get the following corollary from
  \cite[Prop. 9.3]{a-Mora94}.
  For its formulation we
  identify $t_J \in A$ for $J \subseteq S$ with its image in $E$.

  \begin{corollary}
    \label{cor:eos}
      Let $\OMat $ be the set of Gr\"obner circuits of a loopless matroid
      $\Mat$ without parallel elements.
      Then the polynomials 
      $\sum_{\nu = 1}^{\#J} (-1)^{\nu -1} t_{J \setminus \{j_{\nu }\}}$, where
      $J \in \Mat \cap \OMat$, form
      a Gr\"obner basis of the defining ideal of the Orlik-Solomon algebra in $E$.
  \end{corollary} 

  Gr\"obner bases of the defining ideal of the Orlik-Solomon algebra inside
  the exterior algebra have been described
  previously (see for example \cite[Thm. 2.8]{a-Yuzvinsky01}, \cite{a-CordovilForge05},\cite{a-peeva03}). 

\section{Non-Commutative Gr\"obner Basics}
  \label{sec:basics}
  \label{sec:groebner}
  
  Recall that the \textit{degree lexicographic order} or \textit{deglex order} on the monomials in $A$ is the
  total order $\dlo$ such that for two monomials $t_{i_1} \cdots t_{i_k}$ and
  $t_{j_1} \cdots t_{j_l}$ in $A$ we have $t_{i_1} \cdots t_{i_k} \dlo t_{j_1} \cdots t_{j_l}$ 
  if and only if either $k\Sord l$ or $k=l$ and for some $0 \Sordeq h \Sord k$ we have
  $i_1 = j_1, \ldots , i_h = j_h$ and $i_{h+1} \Sord j_{h+1}$.
  Any $\xi \in A$ can uniquely be written
  as a polynomial of the form
  $f = c_1 m_1 +\cdots +c_jm_{j}$,
  for non-commutative monomials $m_j\dlo \cdots m_2\dlo m_1$ in the variables $t_s$, $s \in S$, and 
  ring elements $c_1,\ldots ,c_j\in R\setminus \{0\}$.
  In this polynomial, $c_1m_{1}$ is called the \textit{leading term},
  $c_1$ the \textit{leading coefficient} and
  $m_{1}$ the \textit{leading monomial} of $f$.
  We write $\lt(f)$ for the leading term, $\lm(f)$ for the leading monomial 
  and $\lc(f)$ for the leading coefficient of $f$. 
  The $m_1, \ldots, m_k$ are called the \textit{monomials of} $f$.
  In other words, the leading monomial is the largest monomial among all monomials
  of $f$ with respect to the deglex order. Further, for any monomial $m\in
  A$ there are only finitely many monomials $m'\in A$ such that $m'\dlo m$.

  Let $\cI $ be a set of elements of $A$ with leading
  coefficient $1$. A \textit{reduction} of a polynomial $f\in A$
  modulo $\cI $ is an expression obtained from $f$ by replacing
  the leading monomial $m$ of an element $g\in \cI $,
  appearing as a subword of one of the monomials of $f$,
  by $m-g$. By construction, a reduction does not have
  monomials larger than the leading monomial of $f$.
  For any $f,g\in A$
  we say that $f$ \textit{reduces to} $g$ (modulo $\cI $) and write
  \begin{align}
    \label{eq:red}
    f\reduc{\cI}g
  \end{align}
  if there is a sequence of expressions
  $f=f_0,f_1,\dots ,f_k=g$, where $k\in \ndN _0$, such that
  $f_{i+1}$ is a reduction of $f_i$ for all $i\in \{0,1,\dots,k-1\}$.

  A subset $G$ of a two-sided ideal $I$ in $A$ 
  is called a \textit{Gr\"obner basis} of $I$ if
  the two-sided ideal generated by $\{\lt(g)\,|\,g \in G\}$ coincides with the two-sided
  ideal generated by $\{\lt(f)\,|\,f \in I\}$.

  For two polynomials $f$, $g$ in $A$ with $\lc(f) = \lc(g) = 1$ an 
  $S$-\textit{polynomial of}
  $(f,g)$ is any non-zero expression $m_1fm_2-n_1gn_2\in A$
  for monomials $m_1,m_2,n_1,n_2$ 
  such that
  \begin{align}
    m_1\lm(f)m_2 = n_1\lm(g)n_2.
    \label{eq:Spol}
  \end{align}
  Let $J_{f,g}$ be the submodule of the $A$-bimodule 
  $(A \oplus A) \otimes _R(A \oplus A)$
  generated by the tensors $(m_1,n_1)  \otimes (m_2,n_2)$ for monomials
  $m_1,m_2,n_1,n_2$ for which 
  \eqref{eq:Spol} holds.
  Being generated by tensors of pairs of monomials
  there is a unique inclusionwise minimal set of generators of $J_{f,g}$
  consisting of tensors of pairs of monomials. It is easily seen that any of the
  generators will be of the form 
  $(m,1) \otimes (1,n)$,
  $(1,n) \otimes (m,1)$,
  $(1,n_1) \otimes (1,n_2)$ or
  $(m_1,1) \otimes (m_2,1)$.
  The criterion from the following theorem will be employed in order to derive Theorem~\ref{thm:main}.
 
  \begin{theorem} \label{thm:Gcrit}
     Let $R$ be a field and $I$ a two-sided ideal of $A$.
     A set $\cI := \{ f_1, \ldots, f_r \} \subseteq I$
     is a Gr\"obner basis for $I$ if and only
     if for all $1 \leq i \leq j \leq r$ and for any minimal generator 
     $(m_1,n_1) \otimes (m_2,n_2)$ of $J_{f_i,f_j}$ the corresponding $S$-polynomial of $(f_i,f_j)$
     reduces to $0$ modulo $\cI$.
  \end{theorem}

  It is possible to simplify the Gr\"obner basis criterion in Theorem~\ref{thm:Gcrit}
  by using the following fact \cite[Cor.\,5.8]{a-Mora94}.

  \begin{lemma}
    \label{le:simp}
    Let $f,g\in \cI$. Then the $S$-polynomials of $(f,g)$ corresponding to the
    generators $(\lm (g)\,m,1)\otimes (1,m\,\lm (f))$ and
    $(1,\lm (f)\,m)\otimes (m\,\lm (g),1)$ of $J_{f,g}$, where $m$ is an arbitrary monomial,
    reduce to $0$ modulo $\cI $.
    \label{lem:trivialSpol}
  \end{lemma}

  We will apply Theorem \ref{thm:Gcrit} and Lemma \ref{le:simp} in a situation where $R$ is not
  necessarily a field. But since all our polynomials have leading coefficient $1$ and since 
  all reductions only use coefficients $\pm 1$ the assertions remain valid. 

\section{Technical Lemmas} \label{sec:technical}

  \subsection{General set systems}

  In this section we collect some useful formulas which are valid for
  arbitrary set systems $\Mat $ over $S$.

Generalizing the notation in the introduction, for all
$J\subseteq S$ let
\begin{align*}
  \tm_J =\sum _{I\subseteq J,2\nmid \#I} (-1)^{\ell _J(I)}
       (-q)^{(\#I-1)/2}t_{J\setminus I},\quad
  \tp_J =\sum _{I\subseteq J,2\mid \#I} (-1)^{\ell _J(I)}
       (-q)^{\#I/2}t_{J\setminus I}.
\end{align*}

We start with deriving formulas which are valid in $A$.

\begin{lemma}
  \label{le:tptm}
  Let $J,J',J''\subseteq S$ such that $J=J'\cup J''$ and $j'\Sord j''$
  for all $j'\in J'$, $j''\in J''$.
  Then in $A$ equations
  \begin{align}
    \label{eq:tpdec}
    \tp _J=&\;\tp _{J'}\tp _{J''}+(-1)^{\#J'}q\tm _{J'}\tm _{J''},\\
    \label{eq:tmdec}
    \tm _J=&\;\tm _{J'}\tp _{J''}+(-1)^{\#J'}\tp _{J'}\tm _{J''}
  \end{align}
  hold.
\end{lemma}

\begin{proof}
  We proceed by induction on $\#J'$.
  Since $\tp _\emptyset =1$ and $\tm _\emptyset=0$, the claim holds for
  $J'=\emptyset $. Assume now that $\#J\ge 1$,
  $J'=\{j_1\}$ and $J''=J\setminus \{j_1\}$. Then
  \begin{align*}
    \tp _J=&\;\sum _{I\subseteq J,2\mid \#I} (-1)^{\ell _J(I)}
     (-q)^{\#I/2}t_{J\setminus I}\\
     =&\;\sum _{I\subseteq J,2\mid \#I,j_1\in I} (-1)^{\ell _J(I)}
     (-q)^{\#I/2}t_{J\setminus I}
     +\sum _{I\subseteq J,2\mid \#I,j_1\notin I} (-1)^{\ell _J(I)}
     (-q)^{\#I/2}t_{J\setminus I}\\
     =&\;\sum _{L\subseteq J'',2\nmid \#L} (-1)^{\ell _{J''}(L)}
     (-q)^{(\#L+1)/2}t_{J''\setminus L}
     +\sum _{L\subseteq J'',2\mid \#L} (-1)^{\ell _{J''}(L)+\#L}
     (-q)^{\#L/2}t_{j_1}t_{J''\setminus L}\\
     =&\;-q\tm _{J''}+t_{j_1}\tp _{J''}
     =\tp _{J'}\tp _{J''}-q\tm _{J'}\tm _{J''}.
  \end{align*}
  Similarly,
  \begin{align*}
    \tm _J=&\;\tp _{J''}-t_{j_1}\tm _{J''}
     =\tm _{J'}\tp _{J''}-\tp _{J'}\tm _{J''}.
  \end{align*}
  The rest follows from the induction hypothesis and
  the associativity law of $A$.
\end{proof}

\begin{lemma} \label{le:tmtm}
  Let $K,L\subseteq S$ such that $k\Sord l$ for all $k\in K$, $l\in L$.
  Then in $A$ we have
  \begin{align}
    \tm _K\tm _L=\sum _{i=1}^{\#K} (-1)^{\#K-i}\tm _{L\cup K\setminus \{k_i\}}
    =\sum _{i=1}^{\#L}(-1)^{i-1}\tm _{K\cup L\setminus \{l_i\}}.
    \label{eq:tmtm}
  \end{align}
\end{lemma}

\begin{proof}
  We prove the first equality by induction on $\#K$. If $K=\emptyset $ then
  both sides of the equality are zero. Assume now that $K\not=\emptyset $ and
  let $n=\#K$, $K'=K\setminus \{k_n\}$ and $L'=\{k_n\}\cup L$.
  Then Lemma~\ref{le:tptm} applied three times implies that
  \begin{align*}
    \tm _K\tm _L=&\;(\tm _{K'}t_{k_n}+(-1)^{n-1}\tp _{K'})\tm _L\\
    =&\;\tm _{K'}(-\tm _{L'}+\tp _L)+(-1)^{n-1}\tp _{K'}\tm _L\\
    =&\;-\tm _{K'}\tm _{L'}+\tm _{L\cup K\setminus \{k_n\}}
  \end{align*}
  from which the first equation follows from the induction hypothesis.
  In particular, for $L=\emptyset $ we obtain that
  $0=\sum _{i=1}^{\#J}(-1)^{\#J-i}\tm _{J\setminus \{j_i\}}$ for all
  $J\subseteq S$.
  This implies that the second and the third expression in \eqref{eq:tmtm}
  coincide.
\end{proof}

\begin{lemma}
  \label{le:tmspec1}
  Let $J,L\subseteq S$ such that $n:=\#J\ge 2$, $J\setminus \{j_1,j_n\}=L\setminus
  \{l_1\}$ and $j_1\Sord l_1$.
  Then in $A$ we have
  \begin{align} \label{eq:tmspec1}
    \tm _J-\tm _Lt_{j_n} -\sum _{i=2}^{n-1}(-1)^i\tm _{(J\cup L)\setminus \{j_i\}}
    +(-1)^nt_{j_1}\tm _L =0.
  \end{align}
\end{lemma}

\begin{proof}
  Let $J'=J\setminus \{j_1\}$ and $L'=L\setminus \{l_1\}$.
  Then $\tm _J=\tp _{J'}-t_{j_1}\tm _{J'}$ by Lemma~\ref{le:tptm} and
  \[ \sum _{i=2}^{n-1}(-1)^i\tm _{(J\cup L)\setminus \{j_i\}}=
  (t_{l_1}-t_{j_1})\tm _{J'}-(-1)^n\tm _{(J\cup L)\setminus \{j_n\}} \]
  by Lemma~\ref{le:tmtm}. Thus Equation~\eqref{eq:tmspec1} is equivalent to
  \[
  \tp _{J'}-\tm _Lt_{j_n} -t_{l_1}\tm _{J'}
  +(-1)^n\tm _{(J\cup L)\setminus \{j_n\}}+(-1)^nt_{j_1}\tm _L =0.
  \]
  By Lemma~\ref{le:tptm}, the left hand side of the latter equation
  can be written as
  \begin{align*}
    &(\tp _{L'}t_{j_n}+(-1)^{n-2}q\tm _{L'})-(\tp _{L'}-t_{l_1}\tm _{L'})t_{j_n}
    -t_{l_1}(\tm _{L'}t_{j_n}+(-1)^{n-2}\tp _{L'})\\
    &\quad +(-1)^n(\tp _L-t_{j_1}\tm _L)+(-1)^nt_{j_1}\tm _L \\
    &\;=(-1)^{n-2}q\tm _{L'}+(-1)^{n-1}t_{l_1}\tp _{L'}
    +(-1)^n\tp _L=0
  \end{align*}
  which proves the claim.
\end{proof}

\begin{lemma}
  \label{le:tmspec2}
  Let $J,L\subseteq S$ such that $n:=\#J\ge 2$, $L\setminus
  \{l_1\}=J\setminus \{j_1,j_2\}$ and $j_1\Sord l_1\Sord j_2$.
  Then in $A$ we have
  \begin{align} \label{eq:tmspec2}
    \tm _J-(t_{j_2}t_{l_1}+t_{l_1}t_{j_2}-2q)\tm_{L\setminus \{l_1\}}
    -\tm_{j_1j_2}\tm _L
    +\sum_{i=3}^n(-1)^{i+1}\tm_{(J\cup L)\setminus \{j_i\}}=0.
  \end{align}
\end{lemma}

\begin{proof}
  Let $J'=J\setminus \{j_1\}$ and $L'=L\setminus \{l_1\}$.
  Then
  \[ \sum _{i=3}^n(-1)^{i+1}\tm _{(J\cup L)\setminus \{j_i\}}=
  \tm _{\{j_1,l_1,j_2\}}\tm _{L'} \]
  by Lemma~\ref{le:tmtm}. Applying this and Lemma~\ref{le:tptm} repeatedly,
  the left hand side of Equation~\eqref{eq:tmspec1} becomes
  \begin{align*}
    &\tm _J-(t_{j_2}t_{l_1}+t_{l_1}t_{j_2}-2q)\tm_{L'}
    -(t_{j_2}-t_{j_1})(\tp _{L'}-t_{l_1}\tm _{L'}) 
    +(t_{l_1}t_{j_2}-t_{j_1}t_{j_2}+t_{j_1}t_{l_1}-q)\tm_{L'}\\
    &\qquad =\tm _J
    +(q-t_{j_1}t_{j_2})\tm_{L'}
    -(t_{j_2}-t_{j_1})\tp _{L'}\\
    &\qquad = \tp _{J'}-t_{j_1}\tm _{J'}
    +(q-t_{j_1}t_{j_2})\tm_{L'}
    -(t_{j_2}-t_{j_1})\tp _{L'} =0
  \end{align*}
  which proves the claim.
\end{proof}

  Now we turn to reductions.
  Observe that if $\gMat \subseteq \Mat $ are sets of subsets of $S$ then for
  $t\in A$ we have that 
  if $t$ reduces to zero modulo $\GB{\gMat}$ then
  $t$ reduces to zero modulo $\GB{\Mat }$.
  In particular, if $t$ reduces to zero modulo $\GB{\emptyset }$ then
  $t$ reduces to zero modulo $\GB{\Mat }$.
  Note that in the reduction modulo $\GB{\emptyset }$ only relations 
  \eqref{eq:ts2} and \eqref{eq:trts} are involved.

\begin{lemma}
\label{le:ttJ}
 Let $J\subseteq S$ and let $s\in S$.

\begin{enumerate}
\item Assume that $s\in J$. Then
\begin{align*}
t_s \tp _J \redO q\tm _J,\quad
t_s \tm _J \redO \tp _J.
\end{align*}
\item Assume that $s\notin J$. Let $J'=\{j\in J\,|\,j\Sord s\}$.
Then
\begin{align*}
t_s \tp _J \redO (-1)^{\#J'}\tp _{J\cup \{s\}}+q\tm _J,\quad
t_s \tm _J \redO (-1)^{\#J'-1}\tm _{J\cup \{s\}}+\tp _J.
\end{align*}
\end{enumerate}
\end{lemma}

\begin{remark}
  If $s\Sord j_1$ then in the last expression of Lemma~\ref{le:ttJ}(2) the
  leading term of $t_s\tm_J$ is $t_st_{j_2}\cdots t_{j_{\#J}}$. On the other
  hand, the leading term of both $\tm_{J\cup \{s\}}$ and $\tp _J$ is $t_J$ which is
  larger than $t_st_{j_2}\cdots t_{j_{\#J}}$ with respect to the deglex
  order.
  The reduction formula means that these two leading terms
  cancel and $t_s\tm _J$ reduces modulo $\GB{\emptyset }$ to the remaining expression.
  In fact, due to Equation~\eqref{eq:tmdec} for $\tm _{J\cup \{s\}}$,
  no reduction is needed to obtain the result.
\end{remark}

\begin{proof}
 We proceed by induction on $\#J$. Assume first that $s\Sord j$ for all $j\in J$.
 (This holds in particular if $J=\emptyset$.)
 Then $\tp _{J\cup \{s\}}=t_s\tp _J-q\tm _J$ and
 $\tm _{J\cup \{s\}}=\tp _J-t_s\tm _J$ in $A$ by Lemma~\ref{le:tptm}
 and hence (2) holds in this case.

 Assume now that $s\in J$ and $s\Sordeq j$ for all $j\in J$.
 Let $K=J\setminus \{s\}$. Then
\begin{align*}
 \tp _J=t_s\tp _K-q\tm _K,\quad
 \tm _J=\tp _K-t_s\tm _K
\end{align*}
in $A$ by Lemma~\ref{le:tptm}. Since $t_s^2-q\in \cI _q(\emptyset)$,
it follows that
\begin{align*}
 t_s\tp _J=&\;t_s^2\tp _K-qt_s\tm _K\redO q\tp _K-qt_s\tm _K=q\tm _J,\\
 t_s\tm _J=&\;t_s\tp _K-t_s^2\tm _K\redO t_s\tp _K-q\tm _K=\tp _J
\end{align*}
by Lemma~\ref{le:tptm}. Hence (1) holds in this case.

Finally, assume that $J\not=\emptyset $ and that $s>j_1$. Let $K=J\setminus \{j_1\}$. We assume first that
$s\in J$ and prove (1).
Since $t_st_{j_1}+t_{j_1}t_s-2q\in \cI _q(\emptyset )$,
by Lemma~\ref{le:tptm} and by induction hypothesis we obtain that
\begin{align*}
t_s\tp _J=t_s (t_{j_1}\tp _K-q\tm _K)&\;\redO (-t_{j_1}t_s+2q)\tp _K-qt_s \tm _K\\
&\;\redO -t_{j_1}(q\tm _K)+q\tp _K=q\tm _J.
\end{align*}
Similarly,
\begin{align*}
t_s\tm _J=t_s (\tp _K-t_{j_1}\tm _K)&\;\redO q\tm _K-(-t_{j_1}t_s+2q)\tm _K\\
&\;\redO t_{j_1}\tp _K-q\tm _K=\tp _J.
\end{align*}
A similar argument proves (2).
\end{proof}

The following lemma is a right-handed analogue of the previous result.

\begin{lemma}
\label{le:tJt}
 Let $J\subseteq S$ and let $s\in S$.

\begin{enumerate}
\item Assume that $s\in J$. Then
\begin{align*}
  \tp _J t_s\redO (-1)^{\#J+1}q\tm _J,\quad
  \tm _J t_s \redO (-1)^{\#J+1}\tp _J.
\end{align*}
\item Assume that $s\notin J$. Let
  $J''=\{j\in J\,|\,s\Sord j\}$. Then
\begin{align*}
  \tp _J t_s\redO &\;(-1)^{\#J''}\tp _{J\cup \{s\}}+(-1)^{\#J +1}q\tm _J,\\
  \tm _J t_s\redO &\;(-1)^{\#J''}\tm _{J\cup \{s\}}+(-1)^{\#J+1}\tp _J.
\end{align*}
\end{enumerate}
\end{lemma}

\begin{proof}
  See the proof of Lemma~\ref{le:ttJ}.
\end{proof}
%

\begin{lemma}
  Let $J\subseteq S$ with $J\not=\emptyset $.
  If $\tm _J$ reduces to zero modulo $\GB{\Mat }$,
  then $\tp _J$
  reduces to zero modulo $\GB{\Mat }$.
\label{le:tpred}
\end{lemma}

\begin{proof}
  Let $K=J\setminus \{j_1\}$.
  Lemma~\ref{le:tptm} gives that
  \begin{align*}
    \tp _J&\;=t_{j_1}\tp _K-q\tm _K,\quad
    \tm _J=\tp _K-t_{j_1}\tm _K
  \end{align*}
  in $A$ and the leading term of $\tm _J$ is the leading term of $\tp _K$.
  Thus, since $\tm _J\redO[\Mat ]0$, it follows that
  $\tp _K\redO [\Mat ]t_{j_1}\tm _K$. Hence
  \begin{align*}
    \tp _J=t_{j_1}\tp _K-q\tm _K\redO [\Mat ]t_{j_1}^2\tm _K-q\tm _K
  \end{align*}
  which reduces to zero modulo $\GB{\Mat }$ since
  $t_{j_1}^2-q\in \GB{\Mat }$.
\end{proof}

\begin{lemma}
  Let $J\subseteq S$ and let $s\in S$. Assume that $\tm _J$ reduces to zero
  modulo $\GB{\Mat }$. If $J\not=\emptyset $, $r\Sord s$ for all $r\in J$
  or $\#J\ge 2$, $s\Sord j_2$ then $\tm _{J\cup \{s\}}$
  reduces to zero modulo $\GB{\Mat }$.
\label{le:subtreduction}
\end{lemma}

\begin{proof}
  If $r\Sord s$ for all $r\in J$
  then $\tm _{J\cup \{s\}}=\tm _Jt_s-\tp _J$
  by Lemma~\ref{le:tptm}. Thus
  the first half of the claim holds by Lemma~\ref{le:tpred}.
  
  If $\#J\ge 2$, $s\Sord j_1$ then
  $\tm _{J\cup \{s\}}=\tp _J-t_s\tm _J$ by Lemma~\ref{le:tptm} and again
  the claim holds. If $s=j_1$ then there is nothing to prove.
  Finally, if $j_1\Sord s\Sord j_2$ then let $J'=J\setminus \{j_1\}$.
  Lemma~\ref{le:tptm} gives that
  \begin{align*}
    \tm _{J\cup \{s\}}=&\;\tm _{\{j_1,s\}}\tp _{J'}+\tp_{\{j_1,s\}}\tm _{J'}
    =(t_s-t_{j_1})\tp _{J'}+(t_{j_1}t_s-q)\tm_{J'}\\
    =&\;(t_s-t_{j_1})(\tm_J+t_{j_1}\tm_{J'})+(t_{j_1}t_s-q)\tm_{J'}\\
    =&\;(t_s-t_{j_1})\tm_J+(t_st_{j_1}+t_{j_1}t_s-2q)\tm_{J'}-(t_{j_1}^2-q)
    \tm_{J'}.
  \end{align*}
  This expression reduces to zero modulo $\GB{\Mat }$ which
  proves the remaining claim.
\end{proof}

\subsection{Matroids}

{}From now on let $\Mat $ be the set of circuits of a loopless matroid without parallel
elements on ground set $S$
and let $\OMat $ be a set of Gr\"obner circuits of $\Mat $.

\begin{example}
  A typical example where the set of circuits $\Mat $ of a matroid is not
  sufficient to define a Gr\"obner basis of $\dos _q(\Mat )$ is the following.

  Let $S=\{1,2,3,4\}$ with the usual order and let $\Mat $ be the set
  system consisting of $\{1,2,4\}$. Then
  $$\tm_{1234}=t_2t_3t_4-t_1t_3t_4+t_1t_2t_4-t_1t_2t_3-qt_4+qt_3-qt_2+qt_1$$
  is zero in $\dos _q(\Mat )$ since
  $$\tm _{1234}=-t_3\tm_{124}+\tp_{124}=-t_3\tm_{124}+t_1\tm_{124}$$
  by Lemma~\ref{le:ttJ} and $\tm_{124}=0$ in $\dos_q(\Mat )$.
  However, the leading term of $\tm _{1234}$ cannot be reduced using the
  generators of $I_q(\Mat )$.
\end{example}

Before we prove that $\GB{\OMat }$ is a Gr\"obner basis of $\dos _q(\Mat )$,
we show that $\tm _J$ reduces to zero modulo $\GB{\OMat }$ for all $J\in \Mat
$.

\begin{lemma}
  Let $J,K\subseteq S$ be two dependent sets such that $J\cap K$ is
  independent. Then for all $l \in 
  J \cup K$ the set $(J \cup K) \setminus \{ l \}$ is dependent.
  \label{le:bigcirc}
\end{lemma}

\begin{proof}
  Let $l\in J\cup K$.
  If there is a circuit contained in $J\setminus \{l\}$ or $K \setminus \{l\}$ then
  it is contained in $(J\cup K)\setminus \{l\}$.
  On the other hand, if $C \subseteq J$, $D\subseteq K$ are circuits containing $l$
  then $C\not=D$ since $J\cap K$ is independent.
  Hence by the circuit axiom there is a circuit contained in $(C\cup D)\setminus
  \{l\}$.
\end{proof}

\begin{proposition} \label{pr:depred0}
  For any dependent set $J$, $\tm _J$ reduces to zero modulo $\GB{\OMat }$.
\end{proposition}

\begin{proof}
  We proceed by induction with respect to deglex order.

  Let $J$ be a dependent set. Recall that $\#J\ge 3$.
  If $J\in \OMat $ then $\tm _J$ reduces to zero. In
  particular, this holds if $J$ is the smallest dependent set with respect
  to deglex order.
  
  Assume now that $J\notin \OMat $. Then there exists
  $K\in \OMat $ such that $K\dlo J$ and
  $K\setminus \{k_1\}\sqsubseteq J\setminus \{j_1\}$.
  In particular, $K$ is dependent. We now distinguish several cases according
  to the relations between $k_1$ and the elements of $J$.

  If $k_1<j_1$ then let $L=\{k_1\}\cup (J\setminus \{j_1\})$.
  In this case $L\dlo J$, $K\subset L$ and
  $K\setminus \{k_1\}\sqsubseteq L\setminus \{l_1\}$.
  Hence $\tm _L$ reduces to zero
  modulo $\GB{\OMat }$ by induction hypothesis. If $L\cap J$ is dependent then
  $\tm _{L\cap J}$ reduces to zero by induction hypothesis and hence
  $\tm _J$ reduces to zero by Lemma~\ref{le:subtreduction}. Assume now that
  $L\cap J$ is independent. Then, by Lemma~\ref{le:bigcirc},
  $(J\cup L)\setminus \{j_i\}$
  is dependent for all $i\in \{2,3,\dots ,\#J\}$ and is smaller
  than $J$ with respect to $\dlo $.
  We conclude from Lemma~\ref{le:tmtm} that
  \begin{align*}
    0=\tm _J-\tm _L+\sum _{i=2}^{\#J}(-1)^i\tm _{(J\cup L)\setminus \{j_i\}}
  \end{align*}
  in $A$ and hence $\tm _J$ reduces to zero by induction hypothesis.

  If $k_1=j_1$ then $\tm _J$ reduces to zero by
  using that $\tm _K$ reduces to zero and by repeatedly applying
  Lemma~\ref{le:subtreduction}.

  If $j_1\Sord k_1\Sord j_2$ then $\#K<\#J$ since $K\dlo J$.
  Since $\tm _K$ reduces to zero, by repeatedly applying
  Lemma~\ref{le:subtreduction} we obtain a dependent set $L\subseteq S$
  such that $\#L=\#J-1$, $l_1=k_1$,
  $L\setminus \{l_1\}\sqsubseteq J\setminus \{j_1\}$, and $\tm _L$ reduces to
  zero.
  There are two cases: $j_{\#J}\notin L$ or $j_2\notin L$.
  If $J\cap L$ is dependent then $\tm _J$ reduces to zero by induction
  hypothesis and by Lemma~\ref{le:subtreduction}.
  If $J\cap L$ is independent then $(J\cup L)\setminus \{s\}$ is dependent
  for all $s\in J\cap L$ by Lemma~\ref{le:bigcirc}. Now if $j_{\#J}\notin L$
  then $\tm _J$ reduces to zero by induction hypothesis and by
  Lemma~\ref{le:tmspec1}. Observe that in Lemma~\ref{le:tmspec1} $\tm _J$ and
  $\tm _Lt_{j_{\#J}}$ are the summands with the largest leading term.
  Similarly, if $j_2\notin L$ then $\tm _J$ reduces to zero by induction
  hypothesis and by Lemma~\ref{le:tmspec2}.

  Finally, assume that $j_2=k_1$ or $j_2\Sord k_1$. By
  repeatedly applying Lemma~\ref{le:subtreduction} we obtain that the set
  $L=\{s\in J\,|\,k_1=s\text{ or }k_1\Sord s \}$ is dependent and $L\dlo J$.
  If $l_1\in J$ then $\tm _J$ reduces to zero by Lemma~\ref{le:subtreduction}.
  If $l_1\notin J$ and $J\cap L$ is dependent then again $J$ reduces to zero
  by induction hypothesis and Lemma~\ref{le:subtreduction}.
  In the last case, if $J\cap L$ is independent then $(J\cup L)\setminus \{s\}$
  is dependent for all $s\in L$, $l_1\Sord s$. In this case
  Lemma~\ref{le:tmtm} applied to $J\setminus L$ and $L$
  implies that
  \begin{align*}
    \tm _{J\setminus L}\tm _L=\tm _J+\sum _{i=2}^{\#L}(-1)^{i-1}
    \tm _{(J\cup L)\setminus \{l_i\}}
  \end{align*}
  and the leading term on both sides of the equation is the leading term of
  $\tm _J$. Thus $\tm _J$ reduces to zero by induction hypothesis.
\end{proof}

\section{Independence Statement and Proofs} 
  \label{sec:proofs}

  We first show that $I_q(\Mat)$ is independent of the chosen total order of 
  $S$. 

  \begin{proposition}
    \label{pr:indep}
    For any set system $\Mat $ over $S$,
    the ideal $I_q(\Mat)$ and hence the algebra $\dos _q(\Mat )$ are independent of the 
    total order on $S$.
  \end{proposition}

  \begin{proof}
    We have to show that for any two total orders on $S$ the defining ideals
    of $\dos _q(\Mat )$ coincide.
    Relations~\eqref{eq:ts2}, \eqref{eq:trts} are obviously independent of the chosen
    total order. Relations~\eqref{eq:trts} can be used to reformulate a defining
    relation in \eqref{eq:tJ} in terms of another order.
    We may simplify the problem by looking at orders $\Sord $, $\ll $
    which differ by exchanging two
    neighboring elements $a,b\in S$ with $a\Sord b$, that is, $r\ll s$ for $r,s\in S$
    if and only if either $(r,s)=(b,a)$ or $r\Sord s$, $(r,s)\not=(a,b)$.
    We write $\ell ^<_J(I)$ and $t^<_{J\setminus I}$ and similarly for $\ll $
    to indicate the dependency on the order.

    Let $J\in \Mat $, $I\subseteq J$ with $2\nmid \#I$.
    If $a\notin J$ or $b\notin J$ then 
    $$
      (-1)^{\ell ^\ll _J(I)}=(-1)^{\ell ^<_J(I)},\quad
      (-q)^{(\#I-1)/2}t^\ll_{J\setminus I}=(-q)^{(\#I-1)/2}t^<_{J\setminus I}
    $$
    and hence \eqref{eq:tJ} takes the same form with respect to $\Sord $ and $\ll $.
    It remains to consider the case when $a,b\in J$. We
    prove that in this case the defining relations differ by a sign.
    Then the proof of the proposition is completed.

    Assume that $a\in I$, $b\in J\setminus I$ or $a\in J\setminus I$, $b\in I$. Then
    $$(-1)^{\ell ^\ll _J(I)}=-(-1)^{\ell ^<_J(I)},\quad
    t^\ll _{J\setminus I}=t^<_{J\setminus I}$$
    and therefore $(-1)^{\ell ^\ll _J(I)}(-q)^{(\#I-1)/2}t^\ll _{J\setminus I}=
    -(-1)^{\ell ^<_J(I)}(-q)^{(\#I-1)/2}t^<_{J\setminus I}$.
    Assume now that $a,b\in J\setminus I$ and let $I'=I\cup \{a,b\}$. Then
    \begin{gather*}
      \ell ^\ll _J(I)=\ell ^<_J(I),\quad \ell ^\ll _J(I')=\ell ^<_J(I'),\quad
      t^\ll _{J\setminus I'}=t^<_{J\setminus I'},\\
      t^\ll _{J\setminus I}=t_{j_1}\cdots (t_bt_a)\cdots t_{j_{\#J}}
      =t_{j_1}\cdots (-t_at_b+2q)\cdots t_{j_{\#J}}
      =-t^<_{J\setminus I}+2qt^<_{J\setminus I'},
    \end{gather*}
    and $(-1)^{\ell ^<_J(I)}=(-1)^{\ell ^<_J(I')}$. Hence
    \begin{align*}
      &(-1)^{\ell ^\ll _J(I)}(-q)^{(\#I-1)/2}t^\ll_{J\setminus I}
      +(-1)^{\ell ^\ll _J(I')}(-q)^{(\#I'-1)/2}t^\ll_{J\setminus I'}\\
      &\quad =
      (-1)^{\ell ^<_J(I)}(-q)^{(\#I-1)/2}(-t^<_{J\setminus I}+2qt^<_{J\setminus I'})
      +(-1)^{\ell ^<_J(I)}(-q)^{(\#I-1)/2}(-q)t^<_{J\setminus I'}\\
      &\quad = -(-1)^{\ell ^<_J(I)}(-q)^{(\#I-1)/2}t^<_{J\setminus I}
      +q(-1)^{\ell ^<_J(I)}(-q)^{(\#I-1)/2}t^<_{J\setminus I'}\\
      &\quad = -(-1)^{\ell ^<_J(I)}(-q)^{(\#I-1)/2}t^<_{J\setminus I}
      -(-1)^{\ell ^<_J(I')}(-q)^{(\#I'-1)/2}t^<_{J\setminus I'}.
    \end{align*}
    This is what we wanted to show.
  \end{proof}

  Next we provide the proof of the main result.

  \begin{proof}[Proof of Theorem \ref{thm:main}]

    \noindent {\sf Claim:} If $\Mat$ is a loopless matroid without parallel
    elements then $I_q(\OMat) = I_q(\Mat)$. 

    $\triangleleft$ 
      {}From Proposition \ref{pr:depred0} we deduce that $\tm_J \in I_q(\OMat)$ for 
      all circuits $J$. Hence $I_q(\Mat) \subseteq I_q(\OMat)$. To show equality it
      suffices to show that $\tm_J \in I_q(\Mat)$ for all dependent $J \subseteq S$.
      We prove the assertion by induction on the cardinality of the difference set
      of $J$ and the circuit of largest cardinality contained in it. If the 
      cardinality is $0$ then $J$ itself is a circuit and hence $\tm_J \in I_q(\Mat)$
      by definition. If the cardinality is positive then there is an $s \in J$ such that
      $J \setminus \{s\}$ is dependent and by induction $\tm_{J \setminus \{s\}} \in 
      I_q(\Mat)$. By Lemma \ref{le:ttJ}(2) we can write $\tm_J$ as an $A$ linear
      combination of $\tp_{J \setminus \{s\}}$, $t_s\tm_{J \setminus \{s\}}$ and 
      elements of $I_q(\emptyset) \subseteq I_q(\Mat)$. By Lemma \ref{le:tptm} and
      since $\tm_{J\setminus \{s\}}\in I_q(\Mat)$ it follows that 
      $\tp_{J \setminus \{s\}}\in I_q(\Mat)$ and the assertion follows.
    $\triangleright$

    We complete the proof by showing that $\GB{\Mat}$ is a Gr\"obner basis of $I_q(\OMat) = I_q(\Mat)$.
    For this we verify that the conditions of Theorem \ref{thm:Gcrit} under the simplification provided 
    by Lemma \ref{le:simp} are fulfilled.


    First we have to find minimal generators of the modules $J_{f,g}$,
    where $f,g$ are polynomials
    \eqref{eq:ts2}, \eqref{eq:trts} or $\tm _J$ with $J\in \OMat $.
    According to Lemma~\ref{lem:trivialSpol} we can ignore generators
    $(\lm (g)\,m,1)\otimes (1,m\,\lm (f))$ and
    $(1,\lm (f)\,m)\otimes (m\,\lm (g),1)$ of $J_{f,g}$, where $m$ is an arbitrary monomial.
    Further, since we do not fix an order on the Gr\"obner basis,
    we may restrict ourselves to generators of $J_{f,g}$ of the form
    $(1,m)\otimes (n,1)$ and $(1,m)\otimes (1,n)$.
    Therefore the following cases have to be considered.

    \medskip

    \noindent\textsf{Case} 1.
      $f=t_s^2-q$, $g=t_s^2-q$, $s\in S$.

      The remaining generator of $J_{f,g}$ is
      $(1,t_s)\otimes (t_s,1)$.
      Then $ft_s-t_sg$ is obviously zero.

    \smallskip

    \noindent\textsf{Case} 2.
      $f=t_s^2-q$, $g=t_st_r+t_rt_s-2q$, $r,s\in S$, $r\Sord s$.

      The remaining generator of $J_{f,g}$ is
      $(1,t_s)\otimes (t_r,1)$. The corresponding $S$-polynomial is
      \begin{align*}
        ft_r-t_sg=&\;t_s^2t_r-qt_r-(t_s^2t_r+t_st_rt_s-2qt_s)\\
        \redO &\;-qt_r-(-t_rt_s+2q)t_s+2qt_s\\
        \redO &\;-qt_r+t_rq=0.
      \end{align*}

    \smallskip

    \noindent\textsf{Case} 3.
      $f=t_st_r+t_rt_s-2q$, $g=t_r^2-q$, $r,s\in S$, $r\Sord s$.

      The remaining generator of $J_{f,g}$ is
      $(1,t_s)\otimes (t_r,1)$. The corresponding $S$-polynomial is
      \begin{align*}
        ft_r-t_sg=&\;(t_st_r+t_rt_s-2q)t_r-t_s(t_r^2-q)\\
        \redO &\;t_r(-t_rt_s+2q)-2qt_r+qt_s\\
        \redO &\;-qt_s+qt_s=0.
      \end{align*}

    \smallskip

    \noindent\textsf{Case} 4.
      $f=t_s^2-q$, $g=\tm _J$, $J\in \OMat $, $s=j_2$.

      Let $K=J\setminus \{j_1,j_2\}$.
      The remaining generator of $J_{f,g}$ is
      $(1,t_s)\otimes (t_K,1)$. The corresponding $S$-polynomial is
      \begin{align*}
        ft_K-t_sg=&\;(t_s^2-q)t_K-t_s\tm _J.
      \end{align*}
      By Lemma~\ref{le:ttJ} the expression $t_s\tm _J$ reduces to
      $\tp _J$ modulo $\GB{\emptyset }$. In this reduction
      the leading term $t_s^2t_K$ of $t_s\tm _J$ has to be reduced at one moment to
      $qt_K$. Therefore $t_s\tm _J-(t_s^2-q)t_K$ also reduces to
      $\tp _J$ modulo $\GB{\emptyset }$.
      Thus
      $t_s\tm _J-(t_s^2-q)t_K$ reduces to zero modulo $\GB{\OMat }$
      by Lemma~\ref{le:tpred}.

    \smallskip

    \noindent\textsf{Case} 5.
      $f=\tm _J$, $g=t_s^2-q$, $J\in \OMat $, $s=j_{\#J}$.

      Let $K=J\setminus \{j_1,j_{\#J}\}$.
      The remaining generator of $J_{f,g}$ is
      $(1,t_K)\otimes (t_s,1)$. The corresponding $S$-polynomial is
      \begin{align*}
        ft_s-t_Kg=&\;\tm _Jt_s-t_K(t_s^2-q).
      \end{align*}
      By Lemma~\ref{le:tJt} the expression $\tm _Jt_s$ reduces to
      $(-1)^{\#J+1}\tp _J$ modulo $\GB{\emptyset }$.
      In this reduction
      the leading term $t_Kt_s^2$ of $\tm _Jt_s$ has to be reduced at one moment to
      $qt_K$. Therefore
      $\tm _Jt_s-t_K(t_s^2-q)$
      also reduces to
      $(-1)^{\#J+1}\tp _J$ modulo $\GB{\emptyset }$.
      Thus
      $\tm _Jt_s-t_K(t_s^2-q)$
      reduces to zero modulo $\GB{\OMat }$
      by Lemma~\ref{le:tpred}.

    \smallskip

    \noindent\textsf{Case} 6.
      $f=t_st_r+t_rt_s-2q$, $g=t_rt_p+t_pt_r-2q$, $p,r,s\in S$, $p\Sord r\Sord s$.

      The remaining generator of $J_{f,g}$ is
      $(1,t_s)\otimes (t_p,1)$. The corresponding $S$-polynomial is
      \begin{align*}
        ft_p-t_sg=&\;(t_st_r+t_rt_s-2q)t_p-t_s(t_rt_p+t_pt_r-2q)\\
        =&\;t_rt_st_p-2qt_p-t_st_pt_r+2qt_s\\
        \redO &\;t_r(-t_pt_s+2q)-(-t_pt_s+2q)t_r-2qt_p+2qt_s\\
        \redO &\;-(-t_pt_r+2q)t_s+t_p(-t_rt_s+2q)-2qt_p+2qt_s=0.
      \end{align*}

    \smallskip

    \noindent\textsf{Case} 7.
      $f=t_st_r+t_rt_s-2q$, $g=\tm _J$, $r,s\in S$, $r\Sord s$, $J\in \OMat $,
      $r=j_2$.

      Let $K=J\setminus \{j_1,j_2\}$.
      The remaining generator of $J_{f,g}$ is
      $(1,t_s)\otimes (t_K,1)$. The corresponding $S$-polynomial is
      \begin{align*}
        ft_K-t_sg=&\;(t_st_r+t_rt_s-2q)t_K-t_s\tm _J.
      \end{align*}
      By Lemma~\ref{le:ttJ} the expression $t_s\tm _J$ reduces to
      $\pm \tm _{J\cup \{s\}}+\tp _J$ modulo $\GB{\emptyset }$. In this reduction
      the leading term $t_st_rt_K$ of $t_s\tm _J$ has to be reduced at one moment to
      $(-t_rt_s+2q)t_K$. Therefore $t_s\tm _J-(t_st_r+t_rt_s-2q)t_K$ also reduces to
      $\pm \tm _{J\cup \{s\}}+\tp _J$ modulo $\GB{\emptyset }$.
      Since
      $\tp _J$ reduces to zero modulo $\GB{\OMat }$
      by Lemma~\ref{le:tpred}, it suffices to prove that $\tm _{J\cup \{s\}}$
      reduces to zero modulo $\GB{\OMat }$. The latter holds for $s\Sordeq j_2$ and for
      $j_{\#J}\Sordeq s$ by Lemma~\ref{le:subtreduction} and for $j_2\Sord s\Sord j_{\#J}$
      by (GC).

    \smallskip

    \noindent\textsf{Case} 8.
      $f=\tm _J$, $g=t_st_r+t_rt_s-2q$, $r,s\in S$, $r\Sord s$, $J\in \OMat $,
      $s=j_{\#J}$.

      Let $K=J\setminus \{j_1,s\}$.
      The remaining generator of $J_{f,g}$ is
      $(1,t_K)\otimes (t_r,1)$. The corresponding $S$-polynomial is
      \begin{align*}
        ft_r-t_Kg=&\;\tm _J t_r-t_K(t_st_r+t_rt_s-2q).
      \end{align*}
      The proof is similar to the one in Case 7 and uses Lemma~\ref{le:tJt}.

    \smallskip

    \noindent\textsf{Case} 9.
      $f=\tm _J$, $g=\tm _K$, $J\setminus \{j_1\}\sso K\setminus \{k_1\}$,
      $J,K\in \OMat $, $J\neq K$. This case does not appear by Condition (GC) on
      the elements of $\OMat $.

    \smallskip

    \noindent\textsf{Case} 10.
      $f=\tm _J$, $g=\tm _K$, $J,K\in \OMat $,
      there exists $i\in \{3,\dots ,\# J\}$ such that
      $j_i=k_2$, $j_{i+1}=k_3$, \dots , $j_{\#J}=k_{\#J-i+2}$, $\#J -i+2<\#K$.

      Let $n=\#J-i+2$,
      $$ J'=J\setminus \{j_1\},\quad K'=K\setminus \{k_1\},\quad
      L=\{j\in J'\,|\,j\Sord k_1\},\quad M=\{j\in J'\setminus K'\,|\,k_1\Sord j\}. $$
      Thus the sets $L$, $M$ and $J'\cap K'$ are pairwise disjoint and their union is
      $J'$ (if $k_1\notin J'$) or $J'\setminus \{k_1\}$ (if $k_1\in J'$). We have to
      show that
      \begin{align}
        \tm _Jt_{K'\setminus J'}-t_{J'\setminus K'}\tm _K \redO[\OMat ] 0.
        \label{eq:c10red}
      \end{align}
      We will proceed in several steps, and at some moment we will have to distinguish
      the cases $k_1\in J'$ and $k_1\notin J'$.

      Using Lemma~\ref{le:tptm} and Lemma~\ref{le:tpred} we observe first that
      \begin{align} \label{eq:tmJtK'}
        \tm _Jt_{K'\setminus J'} =\tm _{J\cup K'} \text{ $+$ lower terms which
        reduce to zero modulo $\GB{\OMat}$.}
      \end{align}
      Further, by applying Lemma~\ref{le:ttJ}(2) and Lemma~\ref{le:tpred} we obtain
      that
      \[ t_M\tm _K=\tm _{M\cup K} \text{ $+$ lower terms which
      reduce to zero modulo $\GB{\OMat}$.} \]
      In particular, if $j_1=k_1$ then $M\cup K=J'\cup K=J\cup K'$ and hence
      \eqref{eq:c10red} holds.

      Assume now that $k_1\in J'$. Then $t_{J'\setminus K'}=t_Lt_{k_1}t_M$.
      Lemma~\ref{le:ttJ}(1) gives that
      \[ t_{k_1}\tm _{M\cup K}=\tp _{M\cup K} \text{ $+$ lower terms which
      reduce to zero modulo $\GB{\OMat}$} \]
      and Equation~\eqref{eq:tpdec} implies that
      \[ t_L\tp _{M\cup K}=\tp _{L\cup M\cup K} \text{ $+$ lower terms which
     reduce to zero modulo $\GB{\OMat}$}. \]
     Now, $L\cup M\cup K=(J\cup K)\setminus \{j_1\}$ and
     \begin{align} \label{eq:tmJK'}
      \tm _{J\cup K'}=\tm _{J\cup K}=\tp _{(J\cup K)\setminus \{j_1\}}
      -t_{j_1}\tm _{(J\cup K)\setminus \{j_1\}}
     \end{align}
     by Equation~\eqref{eq:tmdec}. Since $K\subseteq (J\cup K)\setminus \{j_1\}$,
     Proposition~\ref{pr:depred0} yields that
     $\tm _{(J\cup K)\setminus \{j_1\}}$ reduces to zero modulo $\GB{\OMat}$. Thus
     we conclude from \eqref{eq:tmJtK'} and \eqref{eq:tmJK'}
     that \eqref{eq:c10red} holds.

     Finally, assume that $k_1\notin J$. Then
     $t_{J'\setminus K'}=t_Lt_M$. Further,
     \begin{align} \notag
      t_L\tm _{M\cup K}=&\;\tm _{\{j_1\}\cup L}\tm _{M\cup K}
      \text{ $+$ lower terms which reduce to zero modulo $\GB{\OMat}$}\\
      \intertext{by definition of $\tm _{\{j_1\}\cup L}$, and hence}
      t_L\tm _{M\cup K}
      =&\; \sum _{m=1}^{\#(M\cup K)}(-1)^{m-1}\tm _{(J\cup K)\setminus \{(M\cup K)_m\}}
      \text{ $+$ lower terms which reduce to zero}
      \label{eq:tLtmMK}
    \end{align}
    by Lemma~\ref{le:tmtm}, where $(M\cup K)_m$ is the $m$th element of
    $M\cup K$. The summand containing the leading term of the last expression is
    $\tm _{(J\cup K)\setminus \{k_1\}}$ since $(M\cup K)_1=k_1$. Since $J\cap K$ is
    independent by the assumptions $J\cap K=J\cap K'\sso K'$ and $K\in \OMat $,
    Lemma~\ref{le:bigcirc} and Proposition~\ref{pr:depred0}
    imply that all other summands in \eqref{eq:tLtmMK}
    reduce to zero modulo $\GB{\OMat}$. Thus \eqref{eq:c10red} holds in this case.
  \end{proof}

  It remains to provide the proof of Corollary \ref{cor:koszul}.

  \begin{proof}[Proof of Corollary \ref{cor:koszul}]
    \noindent {\sf Claim:} There is a total order on $S$ for which all Gr\"obner circuits are
    of size $3$. 

    $\triangleleft$ We recall the characterization of supersolvable matroids
    given in \cite[Thm. 2.8 (5)]{a-bjoernerziegler91}. There it is shown that for a
    supersolvable matroid $\Mat$ on ground set $S$ the set $S$ can be partitioned into
    subsets $S = S_1 \cup \cdots \cup S_f$ such that for any $1 \leq h \leq f$ and two elements $x,y \in S_h$    
    there is an $1 \leq g < h$ and $z \in S_g$ such that $\{x,y,z\} \in \Mat$ is a circuit.
    Now we choose a total order on $S$ such that for $1 \leq g < h \leq f$ all elements from $S_g$ come
    before $S_h$. Assume that $J$ is a Gr\"obner circuit in this order.  If $J \cap S_h \geq 2$
    for some $1 \leq h \leq f$ then by \cite[Thm. 2.8 (5)]{a-bjoernerziegler91} for any two elements $x,y \in J \cap S_h$
    there is a circuit $K$ of size $3$ such that $K \setminus \{k_1\} = \{ x,y\}$. By choosing two elements $\{x,y\}$ from 
    $J \cap S_h$ for which $\{x,y\} \sqsubseteq J$ we can choose $K$ such that $K \dloeq J$ and
    $K \setminus \{k_1\} \sqsubseteq J \setminus \{j_1\}$. From this it follows by (GC) that $J = K$.
    $\triangleright$

    Now if all Gr\"obner circuits are of size $3$ then the Gr\"obner basis from Corollary \ref{cor:hilbert} is 
    quadratic. Hence by well known facts (see \cite[Sec. 4]{inp-froeberg99}) it follows that 
    $\dos_q(\Mat)$ is Koszul.
  \end{proof}

  \section*{Acknowledgements}
    We thank Graham Denham, Bernd Sturmfels and Alexander Suciu for useful comments on an earlier version of this paper.   

  \bibliographystyle{amsalpha}
  \bibliography{os}
\end{document}